\documentclass[a4paper,twoside,12pt]{article}
\usepackage[a4paper,total={15cm,23.7cm},centering,nohead,ignorefoot,footskip=28pt]{geometry}
\usepackage{amsmath,amssymb,amsthm,array,booktabs,cite,graphicx,pstricks}
\usepackage[all]{xy}

\newcommand{\CL}{\mathsf{CL}}
\newcommand{\IL}{\mathsf{IL}}
\newcommand{\ML}{\mathsf{ML}}
\newcommand{\IT}{\mathsf{IT}}
\newcommand{\NF}{\mathsf{NF}}

\newcommand{\F}{F}

\newcommand{\GText}{G}
\newcommand{\G}[1]{#1^\GText}
\newcommand{\KoText}{\mathit{Ko}}
\newcommand{\Ko}[1]{#1^\KoText}
\newcommand{\KrText}{\mathit{Kr}}
\newcommand{\KrUp}[1]{#1^\KrText}
\newcommand{\KrDown}[1]{#1_\KrText}
\newcommand{\KuText}{\mathit{Ku}}
\newcommand{\KuUp}[1]{#1^\KuText}
\newcommand{\KuDown}[1]{#1_\KuText}

\newcommand{\GN}[1]{#1^{\mathit{GN}}}
\newcommand{\M}[1]{#1^M}
\newcommand{\MText}{M}
\newcommand{\N}[1]{#1^N}
\newcommand{\NText}{N}
\newcommand{\NOneText}{N_1}
\newcommand{\NOne}[1]{#1^{\NOneText}}
\newcommand{\NOneF}[1]{#1^{\NOneText(\F)}}
\newcommand{\NTwoText}{N_2}
\newcommand{\NTwo}[1]{#1^{\NTwoText}}
\newcommand{\NTwoF}[1]{#1^{\NTwoText(\F)}}

\newcommand{\FDText}{\mathit{FD}}
\newcommand{\FD}[1]{#1^\FDText}
\newcommand{\rFDText}{\mathit{FD}'}
\newcommand{\rFD}[1]{#1^{\rFDText}}
\newcommand{\GFDText}{\mathit{GFD}}
\newcommand{\GFD}[1]{#1^\GFDText}

\newcommand{\defEq}{\mathrel{\mathop:}=}
\newcommand{\defEquiv}{\mathrel{\mathop:}\equiv}

\newcommand{\compactVdotsBot}{\dot{\dot{\dot{\bot}}}}
\newcommand{\RAA}{\frac{\substack{\neg A \\ \compactVdotsBot}}{A}}

\newcolumntype{C}{>{\centering\raggedright\arraybackslash}c}
\newcolumntype{M}[1]{>{\centering\raggedright\arraybackslash}m{#1}}

\newcommand{\node}{*-[]{\bullet}}
\newcommand{\phantomNode}{*-[]{\phantom{\bullet}}}
\newcommand{\edge}[1]{\ar@{-}[#1]}
\newcommand{\edgeLabelUp}[2]{\ar@{-}[#1]^{\mbox{#2}}}
\newcommand{\edgeLabelDown}[2]{\ar@{-}[#1]_{\mbox{#2}}}
\newcommand{\edgeTipDown}[2]{\ar@{|-|}@<-1.75ex>[#1]_{\mbox{#2}}}
\newcommand{\edgeTipUp}[2]{\ar@{|-|}[#1]^{\mbox{#2}}}
\newcommand{\labelUp}[1]{\ar@{}^{\mbox{#1}}}
\newcommand{\labelUpStackTwo}[2]{\ar@{}^{\substack{\mbox{#1} \\ \mbox{#2}}}}
\newcommand{\labelDownStackTwo}[2]{\ar@{}_{\substack{\mbox{#1} \\ \mbox{#2}}}}
\newcommand{\labelUpStackThreeA}[3]{\ar@{}^{\substack{\mbox{#1} \\\\ \mbox{#2} \\ \mbox{#3}}}}
\newcommand{\labelUpStackThreeB}[3]{\ar@{}^{\substack{\mbox{#1} \\ \mbox{#2} \\\\ \mbox{#3}}}}
\newcommand{\labelUpStackFour}[4]{\ar@{}^{\substack{\mbox{#1} \\ \mbox{#2} \\\\ \mbox{#3} \\ \mbox{#4}}}}
\newcommand{\labelDown}[1]{\ar@{}_{\mbox{#1}}}
\newcommand{\labelLeft}[1]{\ar@{}^{\mbox{#1}}}
\newcommand{\labelRight}[1]{\ar@{}_{\mbox{#1}}}

\newlength{\myLabelWidth} 
\newenvironment{myDescription}[1]
{\begin{list}{}%
 {\settowidth{\myLabelWidth}{#1}
  \setlength{\leftmargin}{\myLabelWidth}%
  \addtolength{\leftmargin}{\labelsep}
  \setlength{\labelwidth}{\myLabelWidth}}}
{\end{list}}

\theoremstyle{plain}
  \newtheorem{theorem}{Theorem}
  \newtheorem{proposition}[theorem]{Proposition}
  \newtheorem{lemma}[theorem]{Lemma}
\theoremstyle{definition}
  \newtheorem{definition}[theorem]{Definition}
\theoremstyle{remark}
  \newtheorem{remark}[theorem]{Remark}

\begin{document}

\title{Negative translations not intuitionistically equivalent to the usual ones}
\author{Jaime Gaspar\footnote{Arbeitsgruppe Logik, Fachbereich Mathematik, Technische Universit\"at Darmstadt. Schlossgartenstrasse 7, 64289 Darmstadt, Germany. \texttt{mail@jaimegaspar.com}, \texttt{www.jaimegaspar.com}.\newline
I'm grateful to Hajime Ishihara, Ulrich Kohlenbach and Benno van den Berg. This work was financially supported by the Portuguese Funda\c c\~ao para a Ci\^encia e a Tecnologia, grant SFRH/BD/36358/2007.}}
\date{19 March 2011}
\maketitle

\begin{abstract}
  We refute the conjecture that all negative translations are intuitionistically equivalent by giving two counterexamples. Then we characterise the negative translations intuitionistically equivalent to the usual ones.
\end{abstract}

\section{Introduction}

Informally speaking, classical logic $\CL$ is the usual logic in mathematics, and intuitionistic logic $\IL$ is obtained from classical logic by omitting:
\begin{itemize}
  \item reductio ad absurdum $\RAA$;
  \item law of excluded middle $A \vee \neg A$;
  \item law of double negation $\neg\neg A \to A$.
\end{itemize}
In this sense, $\IL$ is a weakening of $\CL$, that is $\IL$ proves less theorems than $\CL$.

At first sight it seems that $\IL$ is just poorer than $\CL$. However, there is a gain in moving from $\CL$ to $\IL$: the theorems of $\IL$ have nicer properties. The main properties gained are
\begin{itemize}
  \item disjunction property: if $\IL \vdash A \vee B$, then $\IL \vdash A$ or $\IL \vdash B$\\
  (where $A$ and $B$ are sentences);
  \item existence property: if $\IL \vdash \exists x A(x)$, then $\IL \vdash A(t)$ for some term $t$\\
  (where $\exists x A$ is a sentence).
\end{itemize}
Arguably, these two properties are the key criteria to say that a logic is constructive.

On the one hand $\IL$ is weaker than $\CL$, on the other hand $\IL$ is constructive while $\CL$ is not. Given these differences, it is surprising that $\CL$ can be faithfully embedded in $\IL$ by the so-called negative translations into $\IL$. Negative translations into $\IL$ are functions $\NText$, mapping a formula $A$ to a formula $\N A$, that:
\begin{itemize}
  \item embed $\CL$ into $\IL$, that is $\CL \vdash A \ \Rightarrow \ \IL \vdash \N A$;
  \item are faithful, that is $\CL \vdash \N A \leftrightarrow A$.
\end{itemize}
The image of the usual negative translations is (essentially) the negative fragment $\NF$, that is the set of all formulas without $\vee$ and $\exists$ and whose atomic formulas are all negated. So $\NF$ is a faithful copy of $\CL$ inside $\IL$. This is pictured in figure \ref{figure:negativeTranslation}.
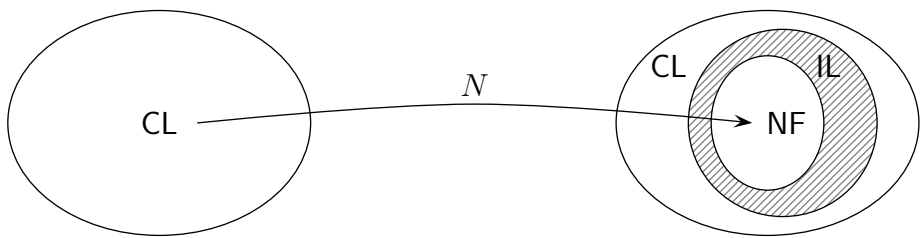
\begin{figure}[h]
  \begin{center}
    \begin{pspicture}(12cm,3cm)
      \psset{linewidth=0.5pt}
      \psellipse(2,1.5)(2,1.5) \rput(2,1.5){$\CL$}
      \psellipse(10,1.5)(2,1.5) \rput(8.7,2.25){$\CL$}
      \pscircle[fillstyle=hlines,hatchcolor=gray,hatchsep=1.5pt,hatchwidth=0.5pt](10.2,1.5){1.25} \rput(10.8,2.25){$\IL$}
      \psellipse[fillstyle=solid](10,1.5)(0.75,0.9) \rput(10.25,1.5){$\NF$}
      \pscurve[fillstyle=none]{->,arrowsize=5pt}(2.5,1.5)(6.15,1.75)(9.8,1.5) \rput(6.15,2){$\NText$}
    \end{pspicture}
    \caption{negative translation $\NText$ into $\IL$ embedding $\CL$ in the fragment $\NF$ of $\IL$.}
  \end{center}
  \label{figure:negativeTranslation}
\end{figure}

There are four negative translations into $\IL$ usually found in the literature (and recently two new ones were presented\cite{FerreiraOliva2011}). They are introduced in table \ref{table:introducingTheFour} and defined (by induction on the structure of formulas) in table \ref{table:definingTheFour}. All these negative translations into $\IL$ are equivalent in $\IL$: given any two of them, say $\MText$ and $\NText$, we have $\IL \vdash \M A \leftrightarrow \N A$. This fact leads to the following conjecture that seems to be almost folklore:
\begin{quote}
  if we rigorously define the notion of a negative translation into $\IL$, then we should be able to prove that all negative translations are equivalent in $\IL$.
\end{quote}
Curiously, this conjecture apparently has never been studied before. In this article we study it, reaching the following conclusions.
\begin{itemize}
  \item The conjecture is false and we give two counterexamples.
  \item The usual negative translations into $\IL$ are characterised by the following two equivalent conditions:
  \begin{itemize}
    \item to translate into $\NF$ in $\IL$, that is $\N A$ is equivalent in $\IL$ to a formula in $\NF$;
    \item to act as the identity on $\NF$ in $\IL$, that is $\IL \vdash \N A \leftrightarrow A$ for all $A \in \NF$.
  \end{itemize}
\end{itemize}

\begin{table}
  \begin{center}
    \begin{tabular}{CCCM{7cm}}
      \toprule
      Year & Name & Symbol & Note\\\midrule
      1925 & Kolmogorov\cite{Kolmogorov1925} & $\KoText$ &\\
      1933 & G\"odel-Gentzen & $\GText$ & One variant by G\"odel\cite{Goedel1933} and another one independently by Gentzen\cite{Gentzen1933}\\
      1951 & Kuroda\cite{Kuroda1951} & $\KuText$ &\\
      1998 & Krivine\cite{Krivine1990} & $\KrText$ & Maybe better attributed to Streicher and Reus\cite{StreicherReus1998}\\\bottomrule
    \end{tabular}
    \caption{the four usual negative translations.}
    \label{table:introducingTheFour}
  \end{center}
\end{table}
\begin{table}
  \begin{center}
    \begin{tabular}{r@{${}\defEquiv{}$}lr@{${}\defEquiv{}$}l}
      \toprule\addlinespace
      $\Ko P$              & $\neg\neg P$ \ ($P\not\equiv \bot$ atomic) & $\G P$              & $\neg\neg P$ \ ($P\not\equiv \bot$ atomic)\\
      $\Ko \bot$           & $\bot$                                     & $\G \bot$           & $\bot$\\
      $\Ko{(A \wedge B)}$  & $\neg\neg(\Ko A \wedge \Ko B)$             & $\G{(A \wedge B)}$  & $\G A \wedge \G B$\\
      $\Ko{(A \vee B)}$    & $\neg\neg(\Ko A \vee \Ko B)$               & $\G{(A \vee B)}$    & $\neg(\neg \G A \wedge \neg \G B)$\\
      $\Ko{(A \to B)}$     & $\neg\neg(\Ko A \to \Ko B)$                & $\G{(A \to B)}$     & $\G A \to \G B$\\
      $\Ko{(\forall x A)}$ & $\neg\neg \forall x \Ko A$                 & $\G{(\forall x A)}$ & $\forall x \G A$\\
      $\Ko{(\exists x A)}$ & $\neg\neg \exists x \Ko A$                 & $\G{(\exists x A)}$ & $\neg \forall x \neg \G A$\\
      \midrule\addlinespace
      $\KuUp A$                & $\neg\neg \KuDown A$           & $\KrUp A$                & $\neg \KrDown A$\\
      $\KuDown P$              & $P$ \ ($P$ atomic)             & $\KrDown P$              & $\neg P$ \ ($P$ atomic)\\
      $\KuDown{(A \wedge B)}$  & $\KuDown A \wedge \KuDown B$   & $\KrDown{(A \wedge B)}$  & $\KrDown A \vee \KrDown B$\\
      $\KuDown{(A \vee B)}$    & $\KuDown A \vee \KuDown B$     & $\KrDown{(A \vee B)}$    & $\KrDown A \wedge \KrDown B$\\
      $\KuDown{(A \to B)}$     & $\KuDown A \to \KuDown B$      & $\KrDown{(A \to B)}$     & $\neg \KrDown A \wedge \KrDown B$\\
      $\KuDown{(\forall x A)}$ & $\forall x \neg\neg \KuDown A$ & $\KrDown{(\forall x A)}$ & $\exists x \KrDown A$\\
      $\KuDown{(\exists x A)}$ & $\exists x \KuDown A$          & $\KrDown{(\exists x A)}$ & $\neg \exists x \neg \KrDown A$\\
      \bottomrule
    \end{tabular}
    \caption{definition of the four usual negative translations.}
    \label{table:definingTheFour}
  \end{center}
\end{table}

\section{Notions}

In the rest of this article, $\CL$ denotes the pure first order classical predicate logic based on $\bot$, $\wedge$, $\vee$, $\to$, $\forall$ and $\exists$ (where $\neg A \defEquiv A \to \bot$, $A \leftrightarrow B \defEquiv (A \to B) \wedge (B \to A)$ and $\equiv$ denotes syntactical equality) and $\IL$ and $\ML$ denote its intuitionistic and minimal counterparts, respectively. All formulas considered belong to the common language of $\CL$, $\IL$ and $\ML$. To save parentheses we adopt the convention that $\forall$ and $\exists$ bind stronger than $\wedge$ and $\vee$, which in turn bind stronger than $\to$.

Let us start by motivating our definition of a negative translation.

The main feature of any negative translation $\NText$ into $\IL$ is embedding $\CL$ into $\IL$ in the sense of $\CL \vdash A \ \Rightarrow \ \IL \vdash \N A$. We can be even more ambitious and ask for (1)~$\CL + \Gamma \vdash A \ \Rightarrow \ \IL + \N \Gamma \vdash \N A$ where $\Gamma$ is any set of formulas and $\N \Gamma \defEq \{\N A : A \in \Gamma\}$.

But embedding $\CL$ into $\IL$ alone does not seem to capture our intuitive notion of a negative translation. For example, it includes the trivial example $\N A \defEquiv \neg\bot$. The problem with this example is that the meaning of $\N A$ is unrelated to the meaning of $A$. So require that a negative translation do not change the meaning of formulas, that is (2)~$\N A \leftrightarrow A$. This equivalence must not be taken in $\IL$ or $\ML$, otherwise from (1) and (2) we would get $\CL = \IL$. So we take the equivalence in $\CL$, that is $\CL \vdash A \leftrightarrow \N A$.

\begin{definition}
  Let $\NText$ be a function mapping each formula $A$ to a formula $\N A$.
  \begin{itemize}
    \item The following condition is called \emph{soundness theorem into $\IL$} ($\ML$) \emph {of $\NText$}: for all formulas $A$ and for all sets $\Gamma$ of possibly open formulas, we have the implication $\CL + \Gamma \vdash A \ \Rightarrow \ \IL + \N \Gamma \vdash \N A$ (respectively, $\CL + \Gamma \vdash A \ \Rightarrow \ \ML + \N \Gamma \vdash \N A$).
    \item The following condition is called \emph{characterisation theorem of $\NText$}: for all formulas $A$ we have $\CL \vdash \N A \leftrightarrow A$.
    \item We say that $\NText$ is a \emph{negative translation into $\IL$} ($\ML$) if and only if both the soundness theorem into $\IL$ (respectively, $\ML$) of $\NText$ and the characterisation theorem of $\NText$ hold.
  \end{itemize}
\end{definition}

\begin{remark}
  The soundness theorem into $\ML$ of $\NText$ implies the soundness theorem into $\IL$ of $\NText$. So a negative translation into $\ML$ is in particular a negative translation into $\IL$.
\end{remark}

The conjecture that concerns us mentions equivalence in $\IL$. For definiteness, we write down exactly what we mean by this.

\begin{definition}
  We say that two negative translations $\MText$ and $\NText$ are \emph{equivalent in $\IL$} ($\ML$) if and only if for all formulas $A$ we have $\IL \vdash \M A \leftrightarrow \N A$ (respectively, $\ML \vdash \M A \leftrightarrow \N A$).
\end{definition}

Later on we will see that what characterises the usual negative translations into $\IL$ are two properties related to $\NF$. Again for definiteness we write down the definition of $\NF$ and of the two properties.

\begin{definition}
  \label{definition:negativeTranslation}
  The \emph{negative fragment} $\NF$ is the set of formulas inductively generated by:
  \begin{itemize}
    \item $\bot \in \NF$;
    \item if $P$ is an atomic formula, then $\neg P \in \NF$;
    \item if $A,B \in \NF$, then $A \wedge B,A \to B,\forall x A \in \NF$.
  \end{itemize}
\end{definition}

\begin{definition}
  Let $\NText$ be a negative translation into $\IL$.
  \begin{itemize}
    \item We say that $\NText$ \emph{translates into $\NF$ in $\IL$} ($\ML$) if and only if for all formulas $A$ there exists a $B \in \NF$ such that $\IL \vdash \N A \leftrightarrow B$ (respectively, $\ML \vdash \N A \leftrightarrow B$).
    \item We say that $\NText$ \emph{acts as the identity on $\NF$ in $\IL$} ($\ML$) if and only if for all $A \in \NF$ we have $\IL \vdash \N A \leftrightarrow A$ (respectively, $\ML \vdash \N A \leftrightarrow A$).
  \end{itemize}
\end{definition}

\section{G\"odel-Gentzen negative translation}

We will choose the G\"odel-Gentzen negative translation $\GText$ as a representative of the usual negative translations into $\IL$, so let us take a closer look at it.

We start by motivating the definition of $\GText$. It is known from proof theory that $\CL$ is conservative over $\ML$ with respect to $\NF$, that is (1)~for all $A \in \NF$ we have the implication $\CL \vdash A \ \Rightarrow \ \ML \vdash A$. This suggests us that one way of constructing a negative translation into $\ML$ is to rewrite each formula $A$ as a formula $\N A \in \NF$. By rewriting we mean that $\N A$ still has the same meaning as $A$ in the sense of (2)~$\CL \vdash \N A \leftrightarrow A$. Then (1) would give us the soundness theorem into $\ML$ of $\NText$ (almost, because there is no $\Gamma$) and (2) would give us the characterisation theorem of $\NText$. The natural way of rewriting a formula $A$ as a classically equivalent formula $\N A \in \NF$ (that is having all atomic formulas $P \not\equiv \bot$ negated and using only $\bot$, $\wedge$, $\to$ and $\forall$) is:
\begin{itemize}
  \item rewrite atomic formulas $P \not\equiv \bot$ as $\neg\neg P$;
  \item rewrite $A \vee B$ as $\neg(\neg A \wedge \neg B)$;
  \item rewrite $\exists x A$ as $\neg \forall x \neg A$;
  \item there's no need to rewrite $\bot$, $A \wedge B$, $A \to B$ and $\forall x A$.
\end{itemize}
If we formalise these rewritings as a definition of $\NText$ by induction on the structure of formulas, then we get exactly $\GText$. As a ``tagline'' we can say: $\G A$ is the natural rewriting of $A$ into $\NF$.

Incidentally, G\"odel's and Gentzen's negative translations differ only in the way they translate $A \to B$: G\"odel translates to $\neg(\G A \wedge \neg \G B)$ while Gentzen translates to $\G A \to \G B$. By the above discussion, we find Gentzen's variant more natural and so we adopt it.

Now we turn to the main properties of $\GText$. We can prove that $G$:
\begin{itemize}
  \item is a negative translation into $\ML$;
  \item translates into $\NF$ in $\ML$;
  \item acts as the identity on $\NF$ in $\ML$.
\end{itemize}
We can even prove strengthenings of the second and third properties above:
\begin{itemize}
  \item for all formulas $A$ we have $\G A \in \NF$;
  \item for all formulas $A \in \NF$ we have $\G A \equiv A$\\
  (modulo identifying $\neg\neg\neg P$ with $\neg P$ for atomic formulas $P$).
\end{itemize}
These two strengthenings are specific of $\GText$: they do not hold for $\KoText$, $\KuText$ and $\KrText$.

To finish this section we discuss $\GText$ as a representative of the usual negative translations into $\IL$. We can prove that $\KoText$, $\GText$, $\KuText$ and $\KrText$ are equivalent in $\IL$. So any of them can be taken as a representative of the usual negative translations into $\IL$. We choose to take $\GText$ as a representative due to nice syntactical properties of $\GText$ like $\G{(A \leftrightarrow B)} \equiv \G A \leftrightarrow \G B$ and the two strengthenings above. These properties allow us to work many times with syntactical equalities instead of equivalences, thus avoiding the question of where ($\CL$, $\IL$ or $\ML$) the equivalences are provable.

\section{Two negative translations not intuitionistically equivalent to the usual ones}

Before we present our two counterexamples to the conjecture, let us draw a scale to roughly measure how provable or refutable a formula $\F$ is. This scale will be useful to picture our main theorem about the counterexamples. We draw the scale following this set of instructions.
\begin{itemize}
  \item We plot along an axis all possible pairs of combinations of
  \begin{equation*}
    \CL \vdash \F, \qquad \CL \nvdash \F \text{ and } \CL \nvdash \neg \F, \qquad \CL \vdash \neg \F
  \end{equation*}
  with
  \begin{equation*}
    \IL \vdash \F, \qquad \IL \nvdash \F \text{ and } \IL \nvdash \neg \F, \qquad \IL \vdash \neg \F.
  \end{equation*}
  \item Actually, we do not plot impossible pairs (for example, ``$\CL \vdash \F$ and $\IL \vdash \neg \F$'') and redundant entries in pairs (for example, the entry ``$\CL \vdash \F$'' in the pair ``$\CL \vdash \F$ and $\IL \vdash \F$'').
  \item The plotting is ordered from provability of $\F$ on the left to refutability of $\F$ on the right (for example, ``$\IL \vdash \F$'' stands on the left of ``$\CL \vdash \F$ and $\IL \nvdash F$'' because ``$\IL \vdash \F$'' states a stronger form of provability).
\end{itemize}
The resulting scale is pictured in figure \ref{figure:scale}.
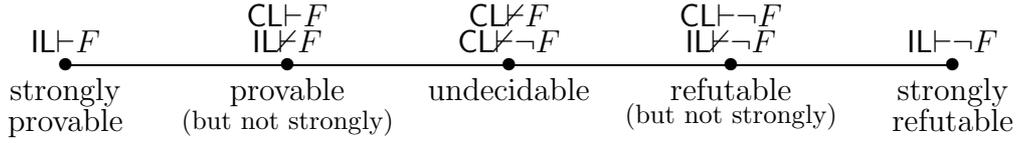
\begin{figure}[h]
  \centerline{
    \xymatrix@C=83pt@L=5pt{
      \node \edge{r} \labelUp{$\IL {\vdash} \F$}\labelDownStackTwo{strongly}{provable} &
      \node \edge{r} \labelUpStackTwo{$\CL {\vdash} \F$}{$\IL {\nvdash} \F$}\labelDownStackTwo{provable}{\footnotesize(but not strongly)} &
      \node \edge{r} \labelUpStackTwo{$\CL {\nvdash} \F$}{$\CL {\nvdash} \neg \F$}\labelDown{undecidable} &
      \node \edge{r} \labelUpStackTwo{$\CL {\vdash} \neg \F$}{$\IL {\nvdash} \neg \F$}\labelDownStackTwo{refutable}{\footnotesize(but not strongly)} &
      \node \labelUp{$\IL {\vdash} \neg \F$}\labelDownStackTwo{strongly}{refutable}
    }
  }
  \caption{scale of provability-refutability.}
  \label{figure:scale}
\end{figure}

Now we present our two counterexamples.
\begin{itemize}
  \item The first counterexample $\NOneText$ is a weakening of $\GText$ obtained by weakening $\G A$ to $\G A \vee F$ (for suitable formulas $\F$).
  \item The second counterexample $\NTwoText$ is a variant of $\GText$ obtained by making $\bot$ in $\G A$ ``less false'' in the sense of replacing $\bot$ by $\F$ in $\G A$, that is $\G A[\F/\bot]$  (again, for suitable $\F$).
\end{itemize}

\begin{definition}
  Fix a formula $\F$. We define two functions $\NOneText$ and $\NTwoText$, mapping formulas to formulas, by
  \begin{itemize}
    \item $\NOne A \defEquiv \G A \vee \F$;
    \item $\NTwo A \defEquiv \G A[\F/\bot]$.
  \end{itemize}
\end{definition}
Since $\NOneText$ and $\NTwoText$ depend on the chosen $\F$, in rigour we should write something like $\NOneF A$ and $\NTwoF A$, but we avoid this cumbersome notation.

We found $\NTwoText$ in an article by Ishihara\cite{Ishihara2000} and in a book chapter by Coquand\cite[section 2.3]{Coquand1997}. Maybe Ishihara drew inspiration from an article by Flagg and Friedman\cite{FlaggFriedman1986} where a similar translation appears. It is even possible that $\NTwoText$ is folklore.

For our two counterexamples to work, we need the formula $\F$ to be classically refutable but intuitionistically acceptable. In the next lemma we prove that there are such formulas $\F$.

\begin{lemma}\mbox{}
  \label{lemma:F}
  \begin{enumerate}
    \item There exists a formula $\F$ such that $\CL \vdash \neg \F$ but $\IL \nvdash \neg \F$.
    \item \label{item:FNotNegative} Any such formula $\F$ is not equivalent in $\IL$ to a formula in $\NF$.
  \end{enumerate}
\end{lemma}

\begin{proof}\mbox{}
  \begin{enumerate}
    \item Let $P$ be an unary predicate symbol. We are going to prove that $\F \equiv \neg \forall x P(x) \wedge \forall x \neg\neg P(x)$ is such that $\CL \vdash \neg \F$ but $\IL \nvdash \neg \F$. Since $\CL \vdash \neg \F$ is obvious, we move on to prove $\IL \nvdash \neg \F$ by showing that the Kripke model $\mathcal K$ from figure \ref{figure:KripkeModelForcingF} forces $\F$.
    \begin{itemize}
      \item $\mathcal K$ forces $\neg \forall x P(x)$ because no node forces $\forall x P(x)$.
      \item $\mathcal K$ forces $\forall x \neg\neg P(x)$ because every node $k$ forces $\neg\neg P(d)$ for all $d$ in its domain $\{0,\ldots,k\}$ since the node $k + 1$ forces $P(d)$.
    \end{itemize}
    \begin{figure}[h]
      \centerline{
        \xymatrix{
          \vdots \\
          \node \edge{u} \labelLeft{$\{0,1,2,3\} \ \phantom{3}$} \labelLeft{$\phantom{(}3$} \labelRight{$\ P(0),P(1),P(2)$} \\
          \node \edge{u} \labelLeft{$\{0,1,2\}   \ \phantom{3}$} \labelLeft{$\phantom{(}2$} \labelRight{$\ P(0),P(1)$}      \\
          \node \edge{u} \labelLeft{$\{0,1\}     \ \phantom{3}$} \labelLeft{$\phantom{(}1$} \labelRight{$\ P(0)$}           \\
          \node \edge{u} \labelLeft{$\{0\}       \ \phantom{3}$} \labelLeft{$\phantom{(}0$}                       \\}}
      \caption{a Kripke model $\mathcal K$ forcing $\neg \forall x P(x) \wedge \forall x \neg\neg P(x)$.}
      \label{figure:KripkeModelForcingF}
    \end{figure}
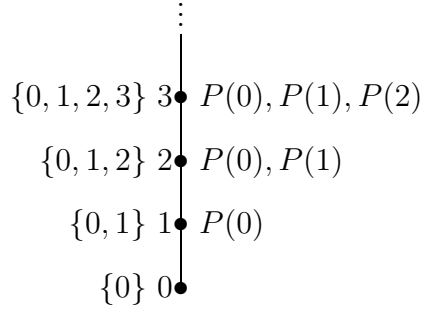
  \item If $\F$ were equivalent in $\IL$ to a formula in $\NF$, then $\neg F$ would also be equivalent in $\IL$ to a formula in $\NF$, so from $\CL \vdash \neg \F$ and the fact that $\CL$ is conservative over $\IL$ with respect to $\NF$ we would get $\IL \vdash \neg \F$, contradicting point 1.\qedhere
  \end{enumerate}
\end{proof}

Now we prove our main theorem giving two counterexamples to the conjecture: $\NOneText$ and $\NTwoText$ are negative translations into $\IL$ (even into $\ML$) not equivalent in $\IL$ to the usual negative translations into $\IL$ (for suitable formulas $\F$). The claims of this theorem are summarised in figure \ref{figure:theoremOnScale}.
\begin{theorem}
  \label{theorem:main}
  The functions $\NOneText$ and $\NTwoText$:
  \begin{enumerate}
    \item \label{item:soundness}have a soundness theorem into $\ML$ for all formulas $\F$;
    \item \label{item:characterisation}have a characterisation theorem if and only if $\CL \vdash \neg\F$;
    \item \label{item:equivalence}are equivalent in $\IL$ to $\GText$ if and only if $\IL \vdash \neg\F$.
    \end{enumerate}
    So, if $\CL \vdash \neg \F$ but $\IL \nvdash \neg F$, then $\NOneText$ and $\NTwoText$ are negative translations into $\ML$ not equivalent in $\IL$ to $\GText$.
  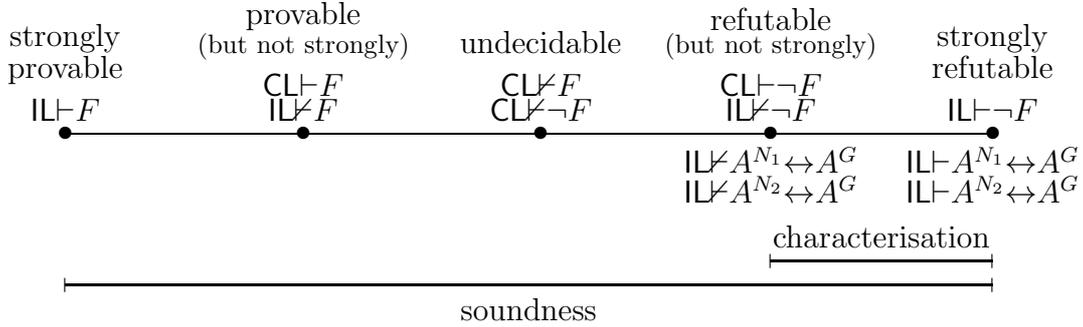
\begin{figure}[h]
    \centerline{
      \xymatrix@C=83pt@L=5pt{
        \node \edge{r} \labelUpStackThreeB{strongly}{provable}{$\IL {\vdash} \F$} &
        \node \edge{r} \labelUpStackFour{provable}{\footnotesize(but not strongly)}{$\CL {\vdash} \F$}{$\IL {\nvdash} \F$} &
        \node \edge{r} \labelUpStackThreeA{undecidable}{$\CL {\nvdash} \F$}{$\CL {\nvdash} \neg \F$} &
        \node \edge{r} \labelUpStackFour{refutable}{\footnotesize(but not strongly)}{$\CL {\vdash} \neg \F$}{$\IL {\nvdash} \neg \F$}{refutable} \labelDownStackTwo{$\IL {\nvdash} \NOne A {\leftrightarrow} \G A$}{$\IL {\nvdash} \NTwo A {\leftrightarrow} \G A$} &
        \node \labelUpStackThreeB{strongly}{refutable}{$\IL {\vdash} \neg \F$} \labelDownStackTwo{$\IL {\vdash} \NOne A {\leftrightarrow} \G A$}{$\IL {\vdash} \NTwo A {\leftrightarrow} \G A$}
        \\\\
        \phantomNode & & & \phantomNode & \phantomNode \edgeTipUp{llll}{soundness} \edgeTipDown{l}{characterisation}
      }
    }
    \caption{theorem \ref{theorem:main} on the scale of provability-refutability.}
    \label{figure:theoremOnScale}
  \end{figure}
\end{theorem}

\begin{proof}\mbox{}
  \begin{enumerate}
    \item Consider an arbitrary formula $\F$.

    First let us consider the case of $\NOneText$. By direct proof, consider an arbitrary set of formulas $\Gamma$ and an arbitrary formula $A$, assume $\CL + \Gamma \vdash A$ and let us prove $\ML + \NOne \Gamma \vdash \NOne A$. Since a proof in $\CL$ of $A$ uses only finitely many formulas $A_1,\ldots,A_n$ from $\Gamma$, then $\CL + A_1 + \cdots + A_n \vdash A$. By the soundness theorem into $\ML$ of $\GText$ we get $\ML + \G A_1 + \cdots + \G A_n \vdash \G A$ (where $\G A_i$ abbreviates $\G{(A_i)}$), that is (1)~$\ML \vdash \G A_1 \wedge \cdots \wedge \G A_n \to \G A$ by the deduction theorem of $\ML$.

    Let us show (2)~$\ML + \G A_1 \vee \F + \cdots + \G A_n \vee \F \vdash \G A \vee \F$. We argue inside $\ML$. Assume $\G A_1 \vee \F,\ldots,\G A_n \vee \F$. Each $\G A_i \vee F$ gives us two cases: the case of $\G A_i$ and the case of $\F$.
    \begin{itemize}
      \item If for some $\G A_i \vee F$ we have the case $\F$, then trivially $\G A \vee \F$.
      \item Otherwise in all $\G A_i \vee F$ we have the case of $\G A_i$, so we have $\G A_1 \wedge \cdots \wedge \G A_n$, thus $\G A$ by (1), therefore trivially $\G A \vee \F$.
    \end{itemize}
    So we have (2) as we wanted. This argument is illustrated for $n = 2$ in figure \ref{figure:argument}.
    \begin{figure}[h]
      \centerline{
        \xymatrix@C=60pt@R=2pt{
                                                             &                                                  & \G A_1 \wedge \G A_2\ar@{->}[r]& \G A \\
                                                             & \G A_2 \vee \F\ar@{-}[ur]^{\G A_2}\ar@{-}[dr]_\F &  \\
                                                             &                                                  & \G A_1 \wedge \F\ar@{->}[r]    & \G A\vee\F \\
          \G A_1 \vee \F\ar@{-}[uur]^{\G A_1}\ar@{-}[ddr]_\F &                                                  &  \\
                                                             &                                                  & \F \wedge \G A_2\ar@{->}[r]    & \G A\vee\F \\
                                                             & \G A_2 \vee \F\ar@{-}[ur]^{\G A_2}\ar@{-}[dr]_\F &  \\
                                                             &                                                  & \F \wedge \F\ar@{->}[r]        & \G A\vee\F
        }
      }
      \caption{argument of $\ML + \G A_1 \vee \F + \cdots + \G A_n \vee \F \vdash \G A \vee \F$ for $n = 2$.}
      \label{figure:argument}
    \end{figure}
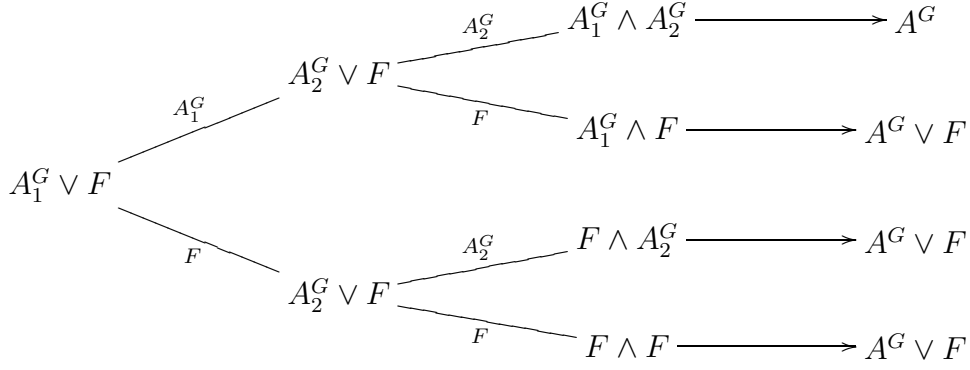

    But (2) is $\ML + \NOne A + \cdots + \NOne A \vdash \NOne A$, so we get $\ML + \NOne \Gamma \vdash \NOne A$, as we wanted.

    Now let us consider the case of $\NTwoText$. By direct proof, consider an arbitrary set of formulas $\Gamma$ and an arbitrary formula $A$, assume $\CL + \Gamma \vdash A$ and let us prove $\ML + \NTwo \Gamma \vdash \NTwo A$. By the soundness theorem into $\ML$ of $\GText$ we get $\ML + \G A_1 + \cdots + \G A_n \vdash \G A$. Since $\bot$ is treated as an arbitrary propositional letter in $\ML$, we can replace $\bot$ by $\F$ getting $\ML + \G A_1[\F/\bot] + \cdots + \G A_n[\F/\bot] \vdash \G A[\F/\bot]$, that is $\ML + \NTwo A_1 + \cdots + \NTwo A_n \vdash \NTwo A$, as we wanted.

    \item First let us consider the case of $\NOneText$.
    \begin{myDescription}{($\Rightarrow$)}
      \item[($\Rightarrow$)] By direct proof, assume that $\NOneText$ has a characterisation theorem and let us prove $\CL \vdash \neg \F$. By the characterisation theorem of $\NOneText$ we have $\CL \vdash \NOne \bot \leftrightarrow \bot$ where $\NOne \bot \equiv \bot \vee \F$, so $\CL \vdash \neg\F$, as we wanted.
      \item[($\Leftarrow$)] By direct proof, assume $\CL \vdash \neg\F$, consider an arbitrary formula $A$ and let us prove $\CL \vdash \NOne A \leftrightarrow A$. By the characterisation theorem of $\GText$ we have $\CL \vdash \G A \leftrightarrow A$. Since $\CL \vdash \neg \F$ by assumption, it makes no difference in $\CL$ to replace $\G A$ by $\G A \vee \F$. So $\CL \vdash \G A \vee \F \leftrightarrow A$, that is $\CL \vdash \NOne A \leftrightarrow A$, as we wanted.
    \end{myDescription}

    Now let us consider the case of $\NTwoText$.
    \begin{myDescription}{($\Rightarrow$)}
      \item[($\Rightarrow$)] Analogous to the case of $\NOneText$.
      \item[($\Leftarrow$)] By direct proof, assume $\CL \vdash \neg \F$, consider an arbitrary formula $A$ and let us prove $\CL \vdash \NTwo A \leftrightarrow A$. By the characterisation theorem of $\GText$ we have $\CL \vdash \G A \leftrightarrow A$. Since $\CL \vdash \neg \F$ by assumption, it makes no difference in $\CL$ to replace $\bot$ by $\F$. So $\CL \vdash \G A[\F/\bot] \leftrightarrow A$, that is $\CL \vdash \NTwo A \leftrightarrow A$, as we wanted.
    \end{myDescription}

    \item First let us consider the case of $\NOneText$.
    \begin{myDescription}{($\Rightarrow$)}
      \item[($\Rightarrow$)] By direct proof, assume that $\NOneText$ and $\GText$ are equivalent in $\IL$ and let us prove $\IL \vdash \neg F$. By the assumption we have $\IL \vdash \NOne \bot \leftrightarrow \G \bot$ where $\NOne \bot \equiv \bot \vee \F$ and $\G \bot \equiv \bot$. So $\IL \vdash \neg\F$, as we wanted.
      \item[($\Leftarrow$)] By direct proof, assume $\IL \vdash \neg\F$, take an arbitrary formula $A$ and let us prove $\IL \vdash \NOne A \leftrightarrow \G A$. By the assumption it makes no difference in $\IL$ to replace $\G A$ by $\G A \vee \F$. So $\IL \vdash \G A \vee \F \leftrightarrow \G A$, that is $\IL \vdash \NOne A \leftrightarrow \G A$, as we wanted.
    \end{myDescription}

    Now let us consider the case of $\NTwoText$.
    \begin{myDescription}{($\Rightarrow$)}
      \item[($\Rightarrow$)] Analogously to the case of $\NOneText$.
      \item[($\Leftarrow$)] By direct proof, assume $\IL \vdash \neg \F$, take an arbitrary formula $A$ and let us prove $\IL \vdash \NTwo A \leftrightarrow \G A$. By the assumption it makes no difference in $\IL$ to replace $\bot$ by $\F$. So $\IL \vdash \G A[\F/\bot] \leftrightarrow \G A$, that is $\IL \vdash \NTwo A \leftrightarrow \G A$, as we wanted.\qedhere
    \end{myDescription}
  \end{enumerate}
\end{proof}

We saw in theorem \ref{theorem:main} that $\NOneText$ and $\NTwoText$ are two counterexamples to the conjecture (for suitable $\F$). Now in proposition \ref{proposition:counterexamplesNotEquivalent} we clarify that these two counterexamples are different (for the same suitable $\F$).

\begin{proposition}
  \label{proposition:counterexamplesNotEquivalent}
  If $\CL \vdash \neg \F$ but $\IL \nvdash \neg \F$, then $\NOneText$ and $\NTwoText$ are not equivalent in $\IL$.
\end{proposition}

\begin{proof}
  By direct proof, assume that $\CL \vdash \neg \F$ but $\IL \nvdash \neg \F$ and let us prove $\IL \nvdash \NOne A \leftrightarrow \NTwo A$. We start by making two observations about Kripke models.
  \begin{enumerate}
    \item There exists a Kripke model $\mathcal K$, with a bottom node, that forces $\neg\F$.

    Let us prove this claim. Since $\CL \vdash \neg \F$ by assumption, any classical model forces $\neg F$. Regarding a classical model as a Kripke model with only one node, we have a Kripke model, with a bottom node, forcing $\neg \F$, as we wanted.

  For example, for the $\F \equiv \neg \forall x P(x) \wedge \forall x \neg\neg P(x)$ used in the proof of lemma \ref{lemma:F}, we can take $\mathcal K$ to be the Kripke model of figure \ref{figure:KripkeModelForcingNotF}.
  \begin{figure}[h]
    \centerline{
      \xymatrix{
        \node \labelLeft{$\{0\}\ $} \labelRight{$\ P(0)$}
      }
    }
    \caption{a Kripke model $\mathcal K$ forcing $\neg\F$ where $\F \equiv \neg \forall x P(x) \wedge \forall x \neg\neg P(x)$.}
    \label{figure:KripkeModelForcingNotF}
  \end{figure}
  \item There exists a Kripke model $\mathcal L$, with a bottom node, that forces $\F$.

  Let us prove this claim. Since $\IL \nvdash \neg\F$ by assumption, there exists a Kripke model $\mathcal L'$ that does not force $\neg\F$, that is some node $n'$ of $\mathcal L'$ does not force $\neg\F$. Then there exists a node $n$ above or equal to $n'$ that forces $\F$. By restricting $\mathcal L'$ to all the nodes above or equal to $n$ we get a Kripke model $\mathcal L$, with bottom node $n$, that forces $\F$, as we wanted.

  For example, for the $\F \equiv \neg \forall x P(x) \wedge \forall x \neg\neg P(x)$ used in the proof of lemma \ref{lemma:F}, we can take $\mathcal L$ to be the Kripke model of figure \ref{figure:KripkeModelForcingF}.
  \end{enumerate}

  Now let us return to our goal: $\IL \nvdash \NOne A \leftrightarrow \NTwo A$. Consider a fresh nullary predicate $Q \not\equiv \bot$. Since $Q$ is fresh and $Q \not\equiv \bot$,
  \begin{itemize}
    \item $\mathcal L$ forces $\neg Q$;
    \item we can force $Q$ in $\mathcal K$;
    \item forcing $Q$ in $\mathcal K$ will not collide with $\mathcal K$ forcing $\neg F$.
  \end{itemize}
  We will show $\IL \nvdash \NTwo Q \to \NOne Q$, where $\NTwo Q \equiv (Q \to \F) \to \F$ and $\NOne Q \equiv \neg\neg Q \vee \F$, by presenting a Kripke model not forcing ($*$)~$((Q \to \F) \to \F) \to \neg\neg Q \vee \F$.

  The base nodes of $\mathcal K$ and $\mathcal L$ have (by definition of Kripke model) non empty domains. We can assume (renaming elements if necessary) that those domains share a common element $d$. Consider the Kripke model $\mathcal M$ from figure \ref{figure:KripkeModelThree} obtained by:
  \begin{itemize}
    \item connecting a fresh bottom node $0$, with domain $\{d\}$, to the bottom nodes of $\mathcal K$ and $\mathcal L$;
    \item for every node $n$ of $\mathcal M$, forcing $Q$ in $n$ if and only if $n$ forces $\neg F$;\\
      or equivalently, forcing $Q$ in $\mathcal K$ but not in $\mathcal L$ and $0$.
  \end{itemize}
  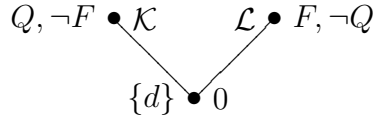
\begin{figure}[h]
    \centerline{
      \xymatrix{
        \node \labelLeft{$Q,\neg F\ $} \labelRight{$\ \mathcal K$} & & \node\labelLeft{$\mathcal L \ $} \labelRight{$\ F,\neg Q$} \\
        & \node \labelLeft{$\{d\} \ $} \labelRight{$\ 0$} \edge{ul} \edge{ur} &
      }
    }
    \caption{a Kripke model $\mathcal M$ not forcing $((Q \to \F) \to \F) \to \neg\neg Q \vee \F$.}
    \label{figure:KripkeModelThree}
  \end{figure}
  Note that $\mathcal M$ is well-defined because:
  \begin{itemize}
    \item the domains of $\mathcal M$ are monotone since $\{d\}$ is contained in the domains of $\mathcal K$ and $\mathcal L$;
    \item the forcing relation in $\mathcal M$ is monotone since $Q$ is forced only in the entire $\mathcal K$.
  \end{itemize}
  Now we argue that $\mathcal M$ does not force ($*$).
  \begin{itemize}
    \item The node $0$ does not force $\neg\neg Q$ because $\mathcal L$ forces $\neg Q$.
    \item The node $0$ does not force $\F$ because $\mathcal K$ forces $\neg\F$.
    \item Let us show that the node $0$ forces $(Q \to \F) \to \F$, that is any node $n$ does not force $Q \to \F$ or forces $\F$. We consider the following three cases.
    \begin{itemize}
      \item If $n$ is in $\mathcal K$, then $n$ does not force $Q \to \F$ because $\mathcal K$ forces $Q$ (by construction of $\mathcal M$) and $\neg\F$.
      \item If $n$ is in $\mathcal L$, then $n$ forces $\F$ because $\mathcal L$ forces $\F$.
      \item If $n$ is $0$, then $n$ does not force $Q \to \F$, otherwise $\mathcal K$ would force $Q \to \F$ and we already saw that this is false.
    \end{itemize}
  \end{itemize}
  We conclude that the node $0$ does not force ($*$), as we wanted.
\end{proof}

As a curiosity, let us see that we have the factorisations $\NTwoText = \FDText \circ \GText$ and $\NTwoText = \rFDText \circ \GText$ of $\NTwoText$ in terms of Friedman-Dragalin translation $\FDText$\cite{Friedman1978,Dragalin1980a} (better known as Friedman's $A$-translation), its refinement $\rFDText$\cite{BergerEtAl2002} and $\GText$. The translation $\FDText$ was used by Friedman and Dragalin to prove that certain intuitionistic theories $\IT$ are closed under Markov rule in the sense of $\IT \vdash \neg\neg \exists x P(x) \ \Rightarrow \ \IT \vdash \exists x P(x)$ where $P(x)$ is an atomic formula.

\begin{definition}
  Fix a formula $\F$.
  \begin{itemize}
    \item \emph{Friedman-Dragalin translation} $\FDText$ maps each formula $A$ to the formula $\FD A$ obtained from $A$ by simultaneously replacing in $A$:
    \begin{itemize}
      \item $\bot$ by $\F$;
      \item all atomic subformulas $P \not\equiv \bot$ by $P \vee \F$.
    \end{itemize}
    \item The \emph{refined Friedman-Dragalin translation} $\rFDText$ maps each formula $A$ to the formula $\rFD A \defEquiv A[\F/\bot]$.
  \end{itemize}
\end{definition}

Naming $\rFDText$ a refinement of $\FDText$ is a little bit misleading, as we explain now. On the one hand, $\rFDText$ simplifies $\FDText$ by dropping the replacement of atomic subformulas $P \not\equiv \bot$ by $P \vee \F$. On the other hand,
\begin{itemize}
  \item $\FDText$ is sound in the sense of $\IL \vdash A \ \Rightarrow \ \ML \vdash \FD A$;
  \item in general $\rFDText$ is sound only in the weaker sense of $\ML \vdash A \ \Rightarrow \ \ML \vdash \rFD A$.
\end{itemize}
So we can say that $\rFDText$ only really refines $\FDText$ on $\ML$, not on $\IL$. This limitation of $\rFDText$ is a problem if we want to apply a Friedman-Dragalin-like translation in $\IL$. But it is not problem if we only want to apply a Friedman-Dragalin-like translation after a negative translation into $\ML$ (not just into $\IL$).

\begin{proposition}[factorisations $\NTwoText = \FDText \circ \GText$ and $\NTwoText = \rFDText \circ \GText$]\mbox{}
  \begin{enumerate}
    \item For all formulas $A$ we have $\ML \vdash \NTwo A \leftrightarrow \FD{(\G A)}$.
    \item For all formulas $A$ we have $\NTwo A \equiv \rFD{(\G A)}$.
  \end{enumerate}
\end{proposition}

\begin{proof}\mbox{}
  \begin{enumerate}
    \item Let us abbreviate $\FD{(\G A)}$ by $\GFD A$. First we recall the definition of $\GText$ writing all negations $\neg A$ in the form $A \to \bot$:
    \begin{align*}
      \G P &\defEquiv (P \to \bot) \to \bot \ \ (P \not\equiv \bot \text{ atomic}), \\
      \G \bot &\defEquiv \bot, \\
      \G{(A \wedge B)} &\defEquiv \G A \wedge \G B, \\
      \G{(A \vee B)} &\defEquiv (\G A \to \bot) \wedge (\G B \to \bot) \to \bot, \\
      \G{(A \to B)} &\defEquiv \G A \to \G B, \\
      \G{(\forall x A)} &\defEquiv \forall x \G A, \\
      \G{(\exists x A)} &\defEquiv \exists x (\G A \to \bot) \to \bot.
    \end{align*}
    Using this we unfold $\NTwoText$ and $\GFDText$ by induction on the structure of formulas:
    \begin{align*}
      \NTwo P &\defEquiv (P \to \F) \to \F \ \ (P \not\equiv \bot \text{ atomic}), \\
      \NTwo \bot &\defEquiv \F, \\
      \NTwo{(A \wedge B)} &\defEquiv \NTwo A \wedge \NTwo B, \\
      \NTwo{(A \vee B)} &\defEquiv (\NTwo A \to \F) \wedge (\NTwo B \to \F) \to \F, \\
      \NTwo{(A \to B)} &\defEquiv \NTwo A \to \NTwo B, \\
      \NTwo{(\forall x A)} &\defEquiv \forall x \NTwo A, \\
      \NTwo{(\exists x A)} &\defEquiv \exists x (\NTwo A \to \F) \to \F, \displaybreak[0] \\[2mm]
      \GFD P &\defEquiv (P \vee \F \to \F) \to \F \ \ (P \not\equiv \bot \text{ atomic}), \\
      \GFD \bot &\defEquiv \F, & \\
      \GFD{(A \wedge B)} &\defEquiv \GFD A \wedge \GFD B, \\
      \GFD{(A \vee B)} &\defEquiv (\GFD A \to \F) \wedge (\GFD B \to \F) \to \F, \\
      \GFD{(A \to B)} &\defEquiv \GFD A \to \GFD B, \\
      \GFD{(\forall x A)} &\defEquiv \forall x \GFD A, \\
      \GFD{(\exists x A)} &\defEquiv \exists x (\GFD A \to \F) \to \F.
    \end{align*}
    Now we prove $\ML \vdash \NTwo A \leftrightarrow \GFD A$ by induction on the structure of formulas. The only non-trivial case is the one of atomic formulas $P \not\equiv \bot$. In this case we argue $\ML \vdash \NTwo P \leftrightarrow \GFD P$ using $\ML \vdash (P \to F) \leftrightarrow (P \vee F \to F)$.
    \item Just note that $\NTwo A$ and $\rFD{(\G A)}$ are both syntactically equal to $\G A[\F/\bot]$: we have we have $\NTwo A \equiv \G A[\F/\bot]$ by definition of $\NTwoText$ and we have $\rFD{(\G A)} \equiv \G A[\F/\bot]$ by definition of $\rFDText$.\qedhere
  \end{enumerate}
\end{proof}

\section{Characterisation of the negative translations intuitionistically equivalent to the usual ones}

There are two properties relative to $\NF$ that the usual negative translations share:
\begin{itemize}
  \item to translate into $\NF$ in $\IL$;
  \item to act as the identity on $\NF$ in $\IL$.
\end{itemize}
We show that these two properties are not shared by $\NOneText$ and $\NTwoText$.

\begin{proposition}
  \label{proposition:counterexamplesNotTranslateNotActAsIdentity}
  If $\CL \vdash \neg\F$ but $\IL \nvdash \neg\F$, then $\NOneText$ and $\NTwoText$:
  \begin{enumerate}
    \item do not translate into $\NF$ in $\IL$;
    \item do not act as the identity on $\NF$ in $\IL$.
  \end{enumerate}
\end{proposition}

\begin{proof}
  We do the proof only for $\NOneText$ since the case of $\NTwoText$ is analogous. By direct proof, assume $\CL \vdash \neg \F$ but $\IL \nvdash \neg \F$ and let us prove points 1 and 2.
  \begin{enumerate}
    \item If $\NOneText$ would translate into $\NF$ in $\IL$, then $\NOne \bot \equiv \bot \vee \F$, which is equivalent in $\IL$ to $\F$, would be equivalent in $\IL$ to a formula in $\NF$, contradicting point \ref{item:FNotNegative} of lemma \ref{lemma:F}.

    \item If $\NOneText$ would act as the identity on $\NF$ in $\IL$, then $\IL \vdash \NOne \bot \leftrightarrow \bot$ (since $\bot \in \NF$) where $\NOne \bot \equiv \bot \vee \F$, so $\IL \vdash \neg \F$, contradicting the assumption $\IL \nvdash \neg \F$.\qedhere
  \end{enumerate}
\end{proof}

Proposition \ref{proposition:counterexamplesNotTranslateNotActAsIdentity} suggests that the two properties relative to $\NF$ may tell the difference between the usual negative translations into $\IL$ and other negative translations into $\IL$. Indeed, now we prove that they characterise the usual negative translations into $\IL$.

\begin{theorem}
  Let $\NText$ be a negative translation into $\IL$ ($\ML$). The following properties are equivalent.
  \begin{enumerate}
    \item $\NText$ is equivalent in $\IL$ (respectively, $\ML$) to $\GText$.
    \item $\NText$ translates into $\NF$ in $\IL$ (respectively, $\ML$).
    \item $\NText$ acts as the identity on $\NF$ in $\IL$ (respectively, $\ML$).
  \end{enumerate}
\end{theorem}

\begin{proof}
  We do the proof only for negative translations into $\IL$ since the case of negative translations into $\ML$ is analogous.
  \begin{myDescription}{$(2 \Rightarrow 3)$}
    \item[$(1 \Rightarrow 2)$] By direct proof, if $N$ is equivalent in $\IL$ to $\GText$, then $N$ translates into $\NF$ in $\IL$ because $\GText$ does so, as we wanted.
    \item[$(2 \Rightarrow 3)$] By direct proof, assume that $\NText$ translates into $\NF$ in $\IL$, consider an arbitrary formula $A \in \NF$ and let us prove $\IL \vdash \N A \leftrightarrow A$. By assumption the formula $\N A$ is equivalent in $\IL$ to a formula in $\NF$, and we have $A \in \NF$, so the formula $\N A \leftrightarrow A$ is equivalent in $\IL$ to a formula in $\NF$. Since $\CL \vdash \N A \leftrightarrow A$ by the characterisation theorem of $\NText$, and since $\CL$ is conservative over $\IL$ with respect to $\NF$, we have $\IL \vdash \N A \leftrightarrow A$, as we wanted.
    \item[$(3 \Rightarrow 1)$] By direct proof, assume that $\NText$ acts as the identity on $\NF$ in $\IL$, consider an arbitrary formula $A$ and let us prove $\IL \vdash \N A \leftrightarrow \G A$. By the characterisation theorem of $\GText$ we have $\CL + A \vdash \G A$ and $\CL + \G A \vdash A$. So by the soundness theorem into $\IL$ of $\NText$ we get $\IL + \N A \vdash \GN A$ and $\IL + \GN A \vdash \N A$ (where $\GN A$ abbreviates $\N{(\G A)}$). Therefore by the deduction theorem of $\IL$ we have (1)~$\IL \vdash \N A \leftrightarrow \GN A$. Since $\G A \in \NF$ by a property of $\GText$, by the assumption we have (2)~$\IL \vdash \GN A \leftrightarrow \G A$. From (1) and (2) we get $\IL \vdash \N A \leftrightarrow \G A$, as we wanted.\qedhere
  \end{myDescription}
\end{proof}

Another property shared by the usual negative translations into $\IL$  is idempotence in $\IL$, that is $\NText \circ \NText = \NText$ in the sense of: $\IL \vdash \N{(\N A)} \leftrightarrow \N A$ for all formulas $A$. Idempotence in $\IL$ is sometimes proved using the properties relative to $\NF$. The proof roughly proceeds like this: if $\NText$ is a negative translation into $\IL$ that (1)~translates into $\NF$ in $\IL$ and (2)~acts as the identity on $\NF$ in $\IL$, then $\N A \in \NF$ by (1), so $\IL \vdash \N{(\N A)} \leftrightarrow \N A$ by (2). (This argument is not rigorous since from (1) we only get that $\N A$ is equivalent in $\IL$ to a formula in $\NF$, not that $\N A \in \NF$.) This relation of idempotence in $\IL$ with the properties relative to $\NF$ can make us suspect that idempotence in $\IL$ also characterises the usual negative translations into $\IL$. But this is not so because, as we will show now, all negative translations into $\IL$ are idempotent in $\IL$ (but not equivalent in $\IL$, as we already saw).

\begin{definition}
  Let $\NText$ be a negative translation into $\IL$. We say that $\NText$ is \emph{idempotent in $\IL$} ($\ML$) if and only if for all formulas $A$ we have $\IL \vdash \N{(\N A)} \leftrightarrow \N A$ (respectively, $\ML \vdash \N{(\N A)} \leftrightarrow \N A$).
\end{definition}

\begin{proposition}
  All negative translations into $\IL$ ($\ML$) are idempotent in $\IL$ (respectively, $\ML$).
\end{proposition}

\begin{proof}
  We do the proof only for negative translations into $\IL$ since the case of negative translations into $\ML$ is analogous.

  Consider an arbitrary negative translation $\NText$ into $\IL$, an arbitrary formula $A$ and let us prove $\IL \vdash \N{(\N A)} \leftrightarrow \N A$. By the characterisation theorem of $\NText$ we have $\CL + \N A \vdash A$ and $\CL + A \vdash \N A$. So by the soundness theorem into $\IL$ of $\NText$ we get $\IL + \N{(\N A)} \vdash \N A$ and $\IL + \N A \vdash \N{(\N A)}$. Then by the deduction theorem of $\IL$ we have $\IL \vdash \N{(\N A)} \leftrightarrow \N A$, as we wanted.
\end{proof}

\section{Conclusion}

The main three points of this article are the following.
\begin{myDescription}{Characterisation\ }
  \item[Conjecture\ ] The fact that the usual negative translations into $\IL$ are equivalent in $\IL$ leads to the conjecture: if we rigorously define the notion of a negative translation into $\IL$, then we should be able to prove that all negative translations are equivalent in $\IL$.
  \item[Refutation\ ] We refuted the conjecture by presenting two counterexamples.
  \item[Characterisation\ ] We characterised the usual negative translations into $\IL$ as being the ones that translate into $\NF$ in $\IL$, or equivalently, that act as the identity on $\NF$ in $\IL$.
\end{myDescription}

\bibliography{References}{}
\bibliographystyle{plain}

\end{document}